\documentclass{article}
\title{Online Optimisation for Online Learning and Control -- From No-Regret to Generalised Error Convergence
}
\author{J. Calliess \\
	OMI, Dept. of Engineering Science,
        University of Oxford, UK.
	}

\date{\today}
\usepackage[usenames]{color} 
\usepackage[pdftex]{graphicx}
\usepackage{amsfonts}
\usepackage{subfigure}
\usepackage{amsmath}
\usepackage{amssymb}
\usepackage{amsthm}

\newtheorem{theorem}{Theorem}{}
{}
\newtheorem{remark}{Remark}{}
{}
{}
\newtheorem{definition}{Definition}{}
\newtheorem{example}{Example}{}
\newtheorem{lemma}{Lemma}{}
{}

\newcommand{\matnorm}[1]{{\left\vert\kern-0.25ex\left\vert\kern-0.25ex\left\vert #1 
    \right\vert\kern-0.25ex\right\vert\kern-0.25ex\right\vert}}
\newcommand{\opnorm}[1]{{\left\vert\kern-0.25ex\left\vert\kern-0.25ex\left\vert #1 
    \right\vert\kern-0.25ex\right\vert\kern-0.25ex\right\vert}}		

\newcommand{\norm}[1]{\left\Vert#1\right\Vert}

\newcommand{\abs}[1]{\left\vert#1\right\vert}

\newcommand{\Real}{\mathbb R}

\newcommand{\nat}{\mathbb N}

\newcommand{\argmin}{\text{argmin}}

\newcommand{\vc}[1]{#1}

\newcommand{\SP}[2]{\ensuremath{\mathbf{\langle} \vc{#1} \mathbf{,} \vc{#2} \mathbf{\rangle}}}

\renewcommand{\d}[1]{\text{ d}#1}

\newcommand{\specrad}{\rho}

\newcommand{\expect}[1]{\ensuremath{ \langle#1\rangle } }

\newcommand{\param}{\ensuremath{\theta}}
\newcommand{\paramspace}{\ensuremath{\Theta}}

\newcommand{\inspace}{\ensuremath{ \mathcal X}}
\newcommand{\outspace}{\ensuremath{ \mathcal Y}}

\newcommand{\fctspace}{\ensuremath{ \mathcal F}}

\newcommand{\hypspace}{\ensuremath{ \hat{\mathcal F}}}

\newcommand{\maxerrn}{\bar{\mathfrak N}} 

\newcommand{\metric}{\, \mathfrak{d}} 

\newcommand{\predf}{\, \mathfrak{  \hat f}} 
\newcommand{\predfn}{\, \mathfrak{  \hat f_n}} 

\newcommand{\seq}[2]{\ensuremath{\Bigl(#1\Bigr)_{#2}}}

\renewcommand{\d}{\ensuremath{\text{ d}}}

\newcommand{\beq}{\begin{equation}}
\newcommand{\eeq}{\end{equation}}

\newcommand{\tinc}{\ensuremath{ \tau}}

\newcommand{\cwip}{\rightsquigarrow}

\newcommand{\cwipt}{\stackrel{t \to \infty}\cwip}

\usepackage[margin=1.2in]{geometry}
\begin{document}

\maketitle

\begin{abstract} This paper presents early work aiming at the development of a new framework for the design and analysis of algorithms for online learning based prediction and control. Firstly, we consider the task of predicting values of a function or time series based on incrementally arriving sequences of inputs by utilising online programming. Introducing a generalisation of standard notions of convergence, we derive theoretical guarantees on the asymptotic behaviour of the prediction accuracies when prediction models are updated by a no-external-regret algorithm.
We prove generalised learning guarantees for online regression and provide an example of how this can be applied to online learning-based control. We devise a model-reference adaptive controller with novel online performance guarantees on tracking success in the presence of a priori dynamic uncertainty. Our theoretical results are accompanied by illustrations on simple regression and control problems.
\end{abstract}

\section{Introduction} Learning is useful in so far it enhances decision making. Often it is necessary to make decisions repeatedly in an uncertain dynamical world. Here, learning can be employed to inform a prediction model that forecasts the consequences of actions.  In the light of information that becomes incrementally available over time, one would hope that a learning algorithm is capable of updating this model online with sufficient rapidity  to facilitate better decisions and to adapt to changing dynamics. And, when actions have real-world impact, it is typically desired to have sufficient theoretical guarantees on the (long-term) dynamics of the system affected by the learning-based decisions. Since decisions will be based on predictions, understanding such dynamics will have to rest on guarantees on the online prediction performance of the learner.

No-regret algorithms and more generally, online programming algorithms can be utilised for fast online learning and prediction of time series (e.g. \cite{Cesa-bianchi:2006,Anava2013}). If a no-external regret bound is achieved then, provided the prediction loss is convex in the parameters, (sub-) optimality guarantees can be given to bound the average prediction errors and the degree of sub-optimality of the parametric predictor whose parameter is the average of all choices of the adapted online predictors' parameters. This has given rise to algorithms that are no longer purely online learning and that decompose learning and prediction into two phases: a learning phase where the online learning method is employed to adapt the learner's parameter online for some time and a subsequent prediction phase making use of the average parameter obtained from the learning phase \cite{Cesa-bianchi:2006,calliess_gordon_aamas08}.

With the aim to avoid such decompositions, we ask a more general question: Without the need to presuppose convexity or having a separate learning phase, does the no-regret property alone allow us to give guarantees on the increased online prediction success of the pure online learning method over time? 
While we show that, without further assumptions, the regret bound alone is not sufficient to ensure vanishing prediction losses in a pure online learning setup in the classical sense of convergence, we do show increasing prediction success in a more general sense: That is, while we will not be able to guarantee that the prediction loss will eventually remain below any upper bound forever, we can guarantee that it will do so for increasingly long durations, provided learning has been allowed to take place sufficiently long.
This gives rise to a new generalised notion of convergence (and thereby of online-learnability) which we will refer to as convergence with \emph{ increasing permanence (i.p.)}. 

Applying online programming algorithms to (parametric) online regression (cf. \cite{Lu16,Hoi2018}), we can then establish i.p.-convergence guarantees on the online prediction loss sequence. Moreover, building on a parametric online regression model to learn and predict a priori uncertain nonlinear dynamics, we derive theory that allows us to devise controllers that are guaranteed to regulate the state of an a priori uncertain nonlinear system to a desired region with increasing permanence. While this property is weaker then traditional desiderata (such as global asymptotic stability or convergence), we argue that it can be easier to achieve by computationally efficient learning-based controllers. Note, in combination with switching control architectures, our results can lead to control designs guaranteed to be eventually stable: For example, we may be satisfied to know our learning-based controller will eventually succeed to move the state into a region of state-space in which another (e.g. linear) controller is capable to take over and achieve stability.

We will introduce our generalised convergence concept in Sec. \ref{sec:cwip} and provide some general theoretical results. In Sec. \ref{sec:OCP4ORegr}, we apply no-regret learning to online regression and provide new learning guarantees based on the results of the previous section. Sec. \ref{sec:mracappl} shows how to combine all preceding results into the design and theoretical analysis of a model-reference adaptive controller with theoretical guarantees on control success. The article will conclude with a brief summary  and an outlook to future work. This is a preprint version of a conference paper that is to be presented at the IFAC- European Control Conference (ECC), 2019.

%

\section{Convergence with Increasing Permanence} \label{sec:cwip}

\begin{definition}[Convergence with increasing permanence]
	\label{definition}
	Let $\mathbb S$ denote a space endowed with metric $\metric : \mathbb S^2 \to \Real$ and consider 
	the sequence $\seq{s_t}{t \in \nat}$ in $\mathbb S$. 
	We say the sequence $(s_t)_{t\in\nat}$  \emph{converges} to $s' \in \mathbb S$  \emph{with increasing permanence (.i.p.)}, written $s_t \cwip s'$, if and only if the sequence remains in any given ball around $s'$ for increasingly long durations. That is, 
	$s_t \cwip s' :\Leftrightarrow $
		$\forall \epsilon >0, D,N \in \nat \exists n \geq N \in \nat \forall i \in \{1,...,D\}: 
		\metric(s',s_{n+i})  \leq \epsilon.$
\end{definition}



It is easy to see that any sequence that is convergent in the classical sense also is i.p.-convergent. 
However, convergence with .i.p. is a more general concept than standard convergence. To see this consider the following example:
\begin{example}\label{ex:incpermconvseq}
	Define the index set $Z_T:= \{ t \in \nat |  t \leq T, \exists n \in \nat: t = 2^n \}$.
	Define the sequence $(q_t)$ with $q_t := \begin{cases} 1, t \in Z_\infty\\
	1/t^2, \text{ otherwise }\end{cases}$. It is easy to show that we have $s_t \cwip 0 \wedge s_t \cwip 0$ but $ s_t \nrightarrow 0$.
\end{example}

Just as with standard convergence, it will be convenient to consider convergence to sets:
\begin{definition}
	A sequence $\seq{s_t}{t \in \nat}$ converges to a set $S$ with increasing permanence, written $s_t  \stackrel{t \to \infty} {\cwip} S$, iff  $\inf_{s \in S} \metric(s,s_t) \stackrel{t \to \infty}{\cwip} 0 $.
\end{definition}

In what is to follow we will consider real-valued sequences and convergence with respect to the canonical metric $ \metric(s,s') = \abs{s-s'}$.

\begin{lemma} \label{lem:sumStconvtpnewconv}
	Assume we are given a non-negative real-valued sequence $(s_t)_{t \in \nat}$ with \newline
	$ S_T := \frac{1}{T} \sum_{t=1}^T s_t \stackrel{T \to \infty}{\to} 0.$
	Then we have: $s_t \cwipt 0$.
%
%
\end{lemma}
%
%
\begin{remark}
	Note, that generally, convergence $ \frac{1}{T} \sum_{t=1}^T s_t \stackrel{T \to \infty}{\to} 0$ does \emph{not} imply classical convergence $s_t \to 0$.
	For a counterexample, consider the sequence  $(q_t)$ from Ex. \ref{ex:incpermconvseq}. It is easy to check that indeed $ \frac{1}{T} \sum_{t=1}^T q_t \stackrel{T \to \infty}{\to} 0$. But, as discussed above, $(q_t)$ does not converge to 0 
	in the classical sense.
\end{remark}
\subsection{Contractive dynamical systems with increasingly permanently bounded disturbances}
In set-point or tracking control, controllers often generate actions with the aim to turn the closed-loop error dynamics of a plant into a stable system with equilibrium $x^*=0$. This means that the closed-loop dynamics can be represented by a contraction with that fixed-point. However, when the dynamics are not know a priori but are learned online, the actual dynamics deviate from a contraction by some time-varying disturbance. If the online learning method succeeds, this disturbance will eventually become increasingly small for increasing durations of time. Motivated by the analysis of such situations, we will next give i.p. convergence guarantees for disturbed contractive systems.

\begin{theorem} \label{thm:cwipcontrdynsys_main} Let $(\inspace, \norm{\cdot})$ be a normed vector space and $\phi : \inspace \to \inspace$ be a contraction with fixed point $x_* \in \inspace$ and Lipschitz constant $\lambda <1$ relative to the metric canonically induced by norm $\norm{\cdot}$.   Let $(y_t)_{t \in \nat} , (d_t)_{t \in \nat}$  be sequences in $\inspace$ satisfying \begin{equation}\label{eq:contractdyn}
	y_{t+1}  = \phi( y_t) + d_t
\end{equation}
for all time steps $t \in \nat_0$. We assume the \emph{disturbances} $d_t$ to be bounded.
Let $r \geq 0$. 
If  $\norm{d_t} \stackrel{n\to \infty}{\cwip} [0,r]$ then we have: $$\norm{ y_t - x_* } \stackrel{t\to \infty}{\cwip} \Bigl [0,  \frac{r}{1- \lambda} \Bigr ].$$
\end{theorem}
For the special case that the disturbances $d_t$ vanish with i.p., the theorem guarantees that the perturbed sequence $(y_t)$ also converges with i.p. to the fixed point $x_*$. 
Another important special case, which we consider below, is when  $\inspace$ is finite-dimensional and $\phi(x) = M x$ for some Schur (i.e. stable) matrix $M$ with $\specrad(M) <1$. It is a special case since then $\phi$ is an eventually contracting map and hence, a contraction relative to some metric $\tilde d$ that is uniformly equivalent to the metric $\metric: (x,x') \mapsto \norm{x-y} $ \cite{Hasselblatt2003}.  In this case we have:
\begin{theorem}\label{thm:stablewipperturbedlindyn_main}
        Let $\norm{\cdot}, \matnorm{\cdot}$ denote the Euclidean and spectral norms, respectively.  Consider the recurrence $x_{t+1} = M x_t + d_t$ $(t \in \nat)$. Let $\sigma =\sum_{i=0}^{\infty}  \matnorm{M^{i}} < \infty$.	
	If the sequence of disturbances $(d_t)$ is bounded and vanishes with increasing permanence up to error $r>0$, i.e. $\norm{d_t} \cwip [0,r] \wedge \exists b \forall t: \norm{d_t} \leq b$ then we have:
	\[\norm{x_t} \stackrel{t \to \infty}{\cwip} [0,\sigma r].\]
\end{theorem}

\section{Programming for Online Regression}
\label{sec:OCP4ORegr}
\subsection{Background: Online (Convex) Programming and No-Regret Algorithms}

In an \emph{online programming (OP)} problem  \cite{gordon99regret,zinkevich:icml2003}, an arbitrary sequence of (stage) \emph{cost}, or \emph{loss}, functions, $(\ell_{t})_{t \in \nat}$, on a domain $F$, is revealed step by step.
At each \emph{stage} or \emph{time step} $t$, one is asked to choose an \emph{action} $\vc a_t$ from a  \emph{feasible set} $F$. The choice is to be made on the basis of information $\mathbb I_t$ about past actions and cost functions up to stage $t-1$.
After the choice is made, information $\mathbb F_t$ about the current cost function $\ell_{t}$ is revealed, and the algorithm suffers a loss amounting to  $\ell_{t}(\vc a_t)$.
The information that the OP algorithm can use to choose the action $\vc a_t$ is summarised in the information set
$\mathbb I_t = \left\{\seq{a_q}{q < t}, \seq{\mathbb F_q}{q<t} \right\}.$
OP problems can arise in varying setups depending on the kind of information available to the algorithm at the time of decision making and the nature of the \emph{loss feedback} it receives after having made the decision.  
%
%
%

To measure the performance of an OP algorithm, we can compare its accumulated loss up to time step $T$ to an estimate of the best cumulative cost attainable against the sequence $(\ell_{t})_{t=1}^T$.
In particular, we estimate the best attainable cost as the cost of the best constant action choice
$\vc a^*_{T}  \in \argmin_{a \in F} \sum_{t=1}^T \ell_t(a)$
chosen with knowledge of the entire sequence $\ell_{1},\ldots,\ell_{T}$. 
This choice leads to a measure called \emph{external regret} or just \emph{regret}:
$\mathcal{R}(T) = \sum_{t=1}^T \ell_{t}(\vc a_t) - \sum_{t=1}^T \ell_{t}(\vc a^*_{T}).$
An algorithm is said to be \emph{no-(external)-regret} if it guarantees that $\max\bigl(0,\mathcal{R}(T) \bigr)$ is not contained in $\Omega(T)$. That is, if there exists a nonnegative sublinear function  $\Delta(T) \in o(T)$ with $\mathcal{R}(T)\leq\Delta(T)$. 
The term \emph{no-regret} is motivated by the fact that the limiting average nonnegative regret of a no-regret algorithm vanishes, i.e., $\limsup_{T \rightarrow \infty} {\max(0,\mathcal{R}(T)})/{T} = 0$.
The sublinear function $\Delta$ is called a \emph{regret bound}.
%
%

%

A prominent special case arises when both the feasible set and the stage loss funtions are convex. The pertaining problem is then called an \emph{Online Convex Programming (OCP)} problem.
Devising no-regret algorithms and bounds for OCP problems is an active area of research in theoretical computer science and machine learning (e.g. \cite{bubeck2014convex,hazan2016optimal,abernethy2012interior,bubeck2015bandit,hu2016bandit,bubeck2014convex}). A particular well-known no-regret algorithm to solve OCPs is the \emph{Greedy Projection} (GP) algorithm \cite{zinkevich:icml2003}, 
which requires feedback about the gradients of the stage loss functions. 
While here, we focus on the case, where noise-free losses are observable, there exist no-regret algorithms for noise corrupted observations \cite{CesaBianchi2010,Belmega2018}.
Recently,  \emph{Online Projected Stochastic Gradient Descent (OPSGD)} \cite{bubeck2014convex} has been proposed which is applicable in the case of pure bandit loss feedback but provides stochastic no-regret bounds that hold true with arbitrarily adjustable probability.
While we focus on the case of deterministic no-external-regret bounds, all our results do extend to such stochastic settings, albeit our guarantees would then merely hold with the pertaining probabilistic confidence provided by the probabilistic no-regret bound.

No-regret bounds and algorithms have been studied and deployed in a great many online learning scenarios, including, among others, time-series prediction in ARMA models \cite{Anava2013}, multi-agent coordination \cite{calliess_gordon_aamas08}, game-theory \cite{shapire:1996,Blumroutingnoregnash:2006,Roughgarden2016}. For a classic text book, the reader is referred to \cite{Cesa-bianchi:2006}, whereas a recent survey can be found in \cite{Hoi2018}.
In what is to follow, we will illustrate the application of our new convergence concept to the particular domain of online regression. In the context of kernel methods, online regression algorithms are briefly touched upon in \cite{Lu16}. However, their theoretical guarantees are limited to online classification.    Applications to other online learning domains provide ample avenues to future work. A recent survey of existing approaches is provided in \cite{Hoi2018}.

\subsection{Application to Online Learning and Prediction}
To keep the exposition concrete, we consider the following online learning and prediction problem:  
Let $(\inspace,\metric_\inspace), (\outspace, \metric_\outspace)$ be two metric spaces.

An algorithm is given the task to predict a time series $\seq{y_t}{t \in \nat} \in \outspace^\nat$ online on the basis of incrementally observing a related time series $\seq{x_t}{t \in \nat} \in \inspace^\nat$ and obtaining feedback after each prediction. We assume there exists a functional relationship 
\begin{equation}
y_t = f(x_t)
\end{equation}
where $f: \inspace \to \outspace$ is some (a priori uncertain) \emph{target function} residing in some class $\fctspace$. To make predictions, the algorithm has access to a  \emph{hypothesis space} $\hypspace$ of \emph{predictors} $\predf(\cdot; \param ): \inspace \to \outspace \,\, (\param \in \paramspace)$ parametrised by \emph{ parameter space} $\paramspace$.
We assume prediction accuracy is measured by a non-negative loss function $$\ell(\cdot;\cdot): \inspace \times \paramspace \to \Real_{\geq 0}$$ which is zero for inputs $x$ and parameters $\param$ for which $f(x) = \predf(x;\param)$.  For example, the loss might quantify the squared distance, i.e.  $\ell(x;\param) =\metric_\outspace\bigl(f(x), \predf(x;\param)\bigr)^2$.
To connect to the OP setup, we can define the stage loss function $$ \ell_t: \param \mapsto \ell(x_t;\param)$$

At the beginning of each stage at time step $t \in \nat$, the algorithm has to pick a parameter $\param_t$ and use it to make a prediction utilising the chosen predictor: $$\hat y_t = \predf(x_t;\param_t).
$$
The parameter $\param_t$ (and hence, the predictor) is chosen on the basis of information set $\mathbb I_t$. For now, we assume  this set contains all previous inputs $(x_1,...,x_{t-1})$ (but might exclude $x_t$) and information about the pertaining prediction losses. (For example, the latter might be given by revelation of the true $y_t = f(x_t)$ after each prediction at time $t$,  from which the loss function $\ell_t$ can be computed.) The stage at time $t$ concludes by revelation of the loss information after the prediction was made and the process enters the next stage at the next time step $t+1$.

A special case of this setting is online regression. Here the task is learning $f$ online and becoming better at predicting its output values based on an i.i.d. input samples that become incrementally available.
\subsubsection{Example: Online regression with Radial Basis Function Neural Networks}
As a concrete example, consider the case where the hypotheses class $\fctspace$ is a set of radial-basis function neural networks (RBFNN) with known structure:

The component functions $f_i: \Real^d \to \Real $ $(i=1,...,d')$ of target function $f$ can be represented by an RBFNN with some weight parameter $\param_i^* \in \Real^m$. These weights are assumed to be unknown a priori but contained in some known convex feasible set $F \subset \Real^m$.
	That is $\forall i \exists \param_i^* \in F: f_i(x) = \SP{\param^*_i}{\phi_i(x)}$. Here, $\phi :\Real^d \to \Real^m$ with each component function $\phi_i(\cdot) = \exp( - \frac{\norm{\cdot -c_i}^2}{\sigma_i^2} )$ being a radial basis function. 


In this representation, the online learning task can be reduced to updating the weights $\param_t:= (\param_{1,t},...,\param_{d,t})$ at time $t$ and insert it into the prediction hypothesis $ \predf(\cdot;\param ) := \SP{\param}{ \phi( \cdot ) }$ to predict the next observation as per $ \hat y_{t} := \predf(x_t; \param_t)$.

To reduce the learning task to solving an OCP, once we have received the true value $y_t = f(x_t)$, we can compute and feed back information about the stage loss function $\ell_t: \param \mapsto  \norm{  \predf_t(x_t;\param) - y_t }_2^2 .$
It is easy to check that, by construction, each $ \ell_t(\cdot)$ is a convex function.
Therefore, by interpreting the weights as actions ($a_t := \param_t$) in the OCP paradigm, the weight hypotheses $\param_t$ can be generated online by utilising any existing no-regret algorithm designed for OCPs.

In the following example simulation, we have done so employing Greedy Projection \cite{zinkevich:icml2003} (refer to Fig. \ref{fig:onlineregrrbfn}).
Here, the task was to predict the real-valued time series $y_t =f(x_t)= \SP{\param^*}{\phi(x_t)}$ where $x_t,y_t \in \Real, \forall t$ and $\param^* \in \Real^4$ was drawn at random and we chose RBF centres $c_1= -1,c_2 = -1/3,c_3=1/3,c_4=1$ and length scale parameters $\sigma_1=-1,\sigma_2=-1/3,\sigma_3=1/3,\sigma_4 = 1$. At the start of the online prediction task the initial hypothesis parameter $\param_1$ was drawn at random. 
Subsequently we generated the time series $y_t$ by sampling the sequence $x_t$ uniformly i.i.d. at random from the interval [-2,2] and computing $y_t = f(x_t)$.
The updates of the $\param_t$ were computed with Greedy Projection (GP). Requiring gradient information of the stage  loss functions,  we fed GP the stage prediction loss gradient $\nabla_{\param_t} \ell_t(\param_t) = 2 (\predf(x_t;\param_t)) \phi(x_t)$ after it had computed a new weight $\param_t$ at each time step $t =1,...,100$.
\begin{figure*}[t]
	\label{fig:onlineregrrbfn}
	\centering
	\begin{tabular}{cccc}
		\includegraphics[width=0.23\textwidth]{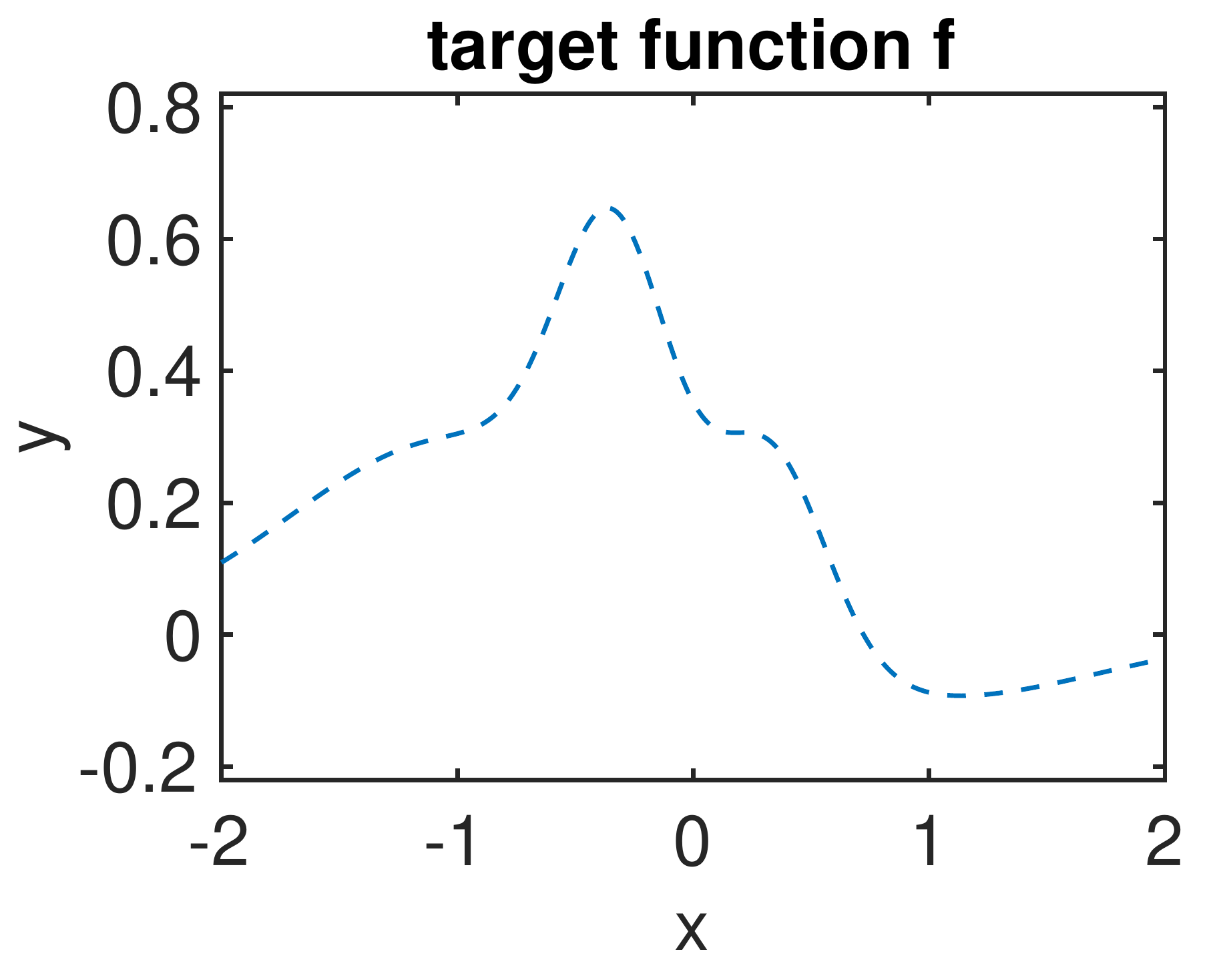}
		&\includegraphics[width=0.23\textwidth]{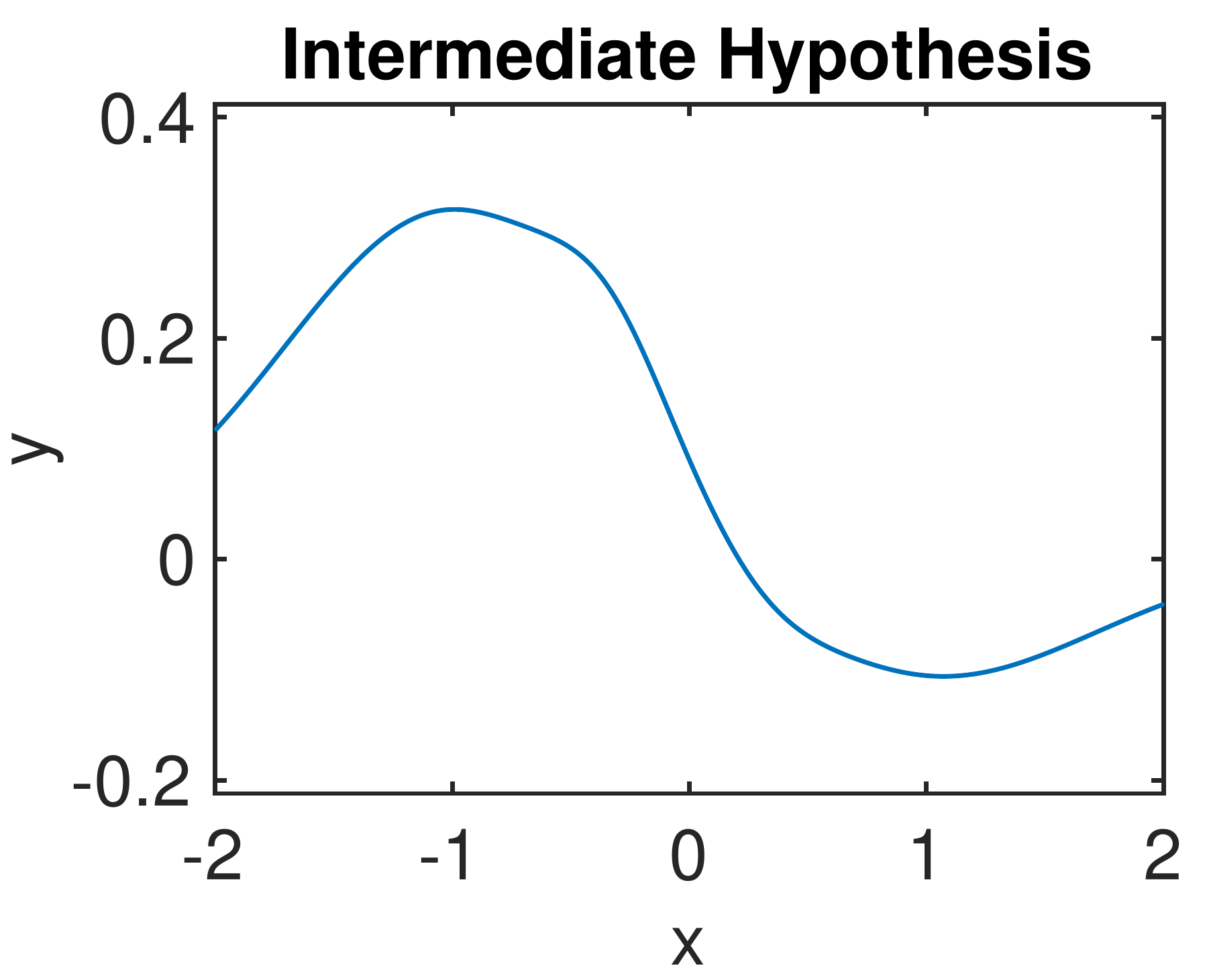} 
		&\includegraphics[width=0.23\textwidth]{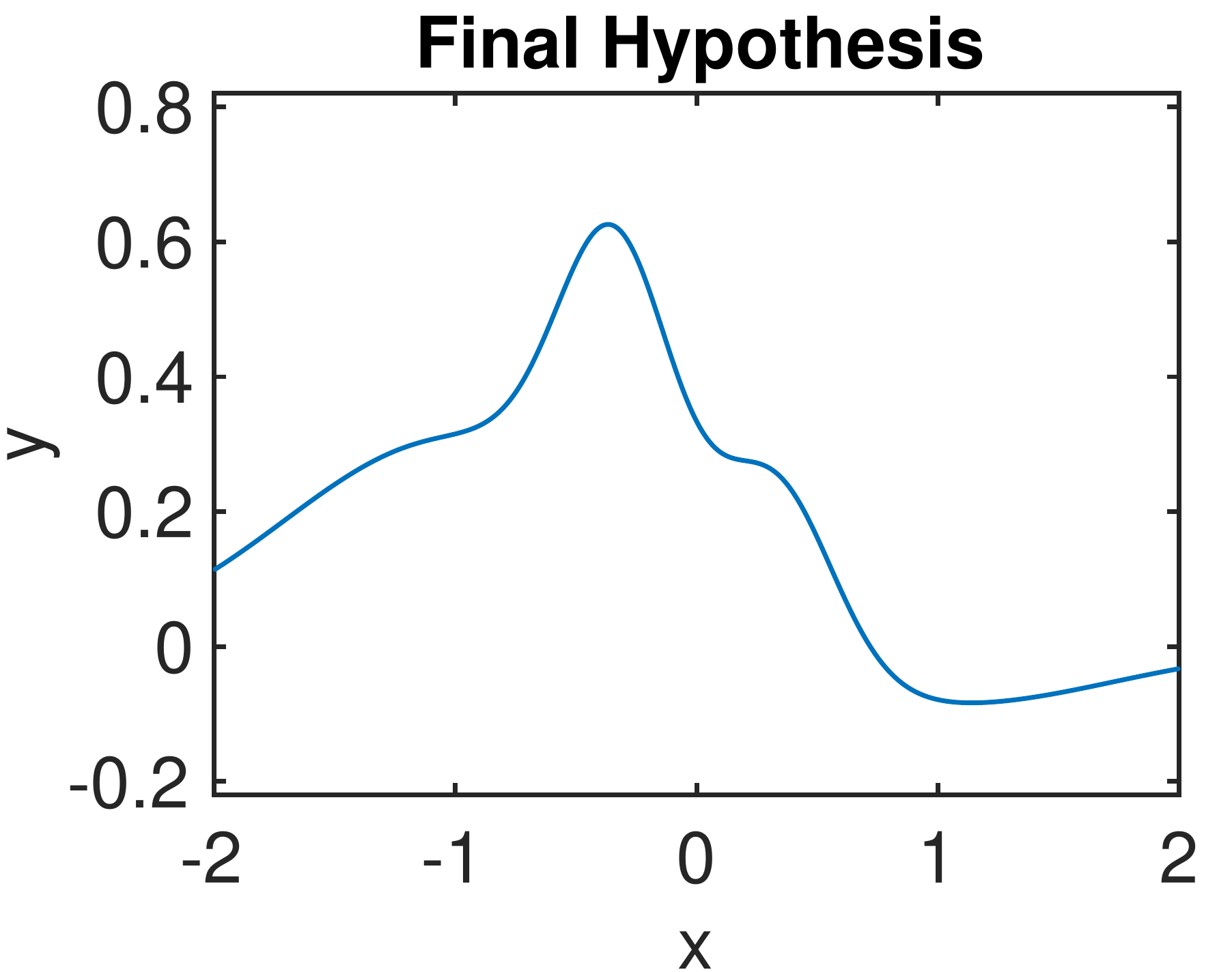}
		&\includegraphics[width=0.23\textwidth]{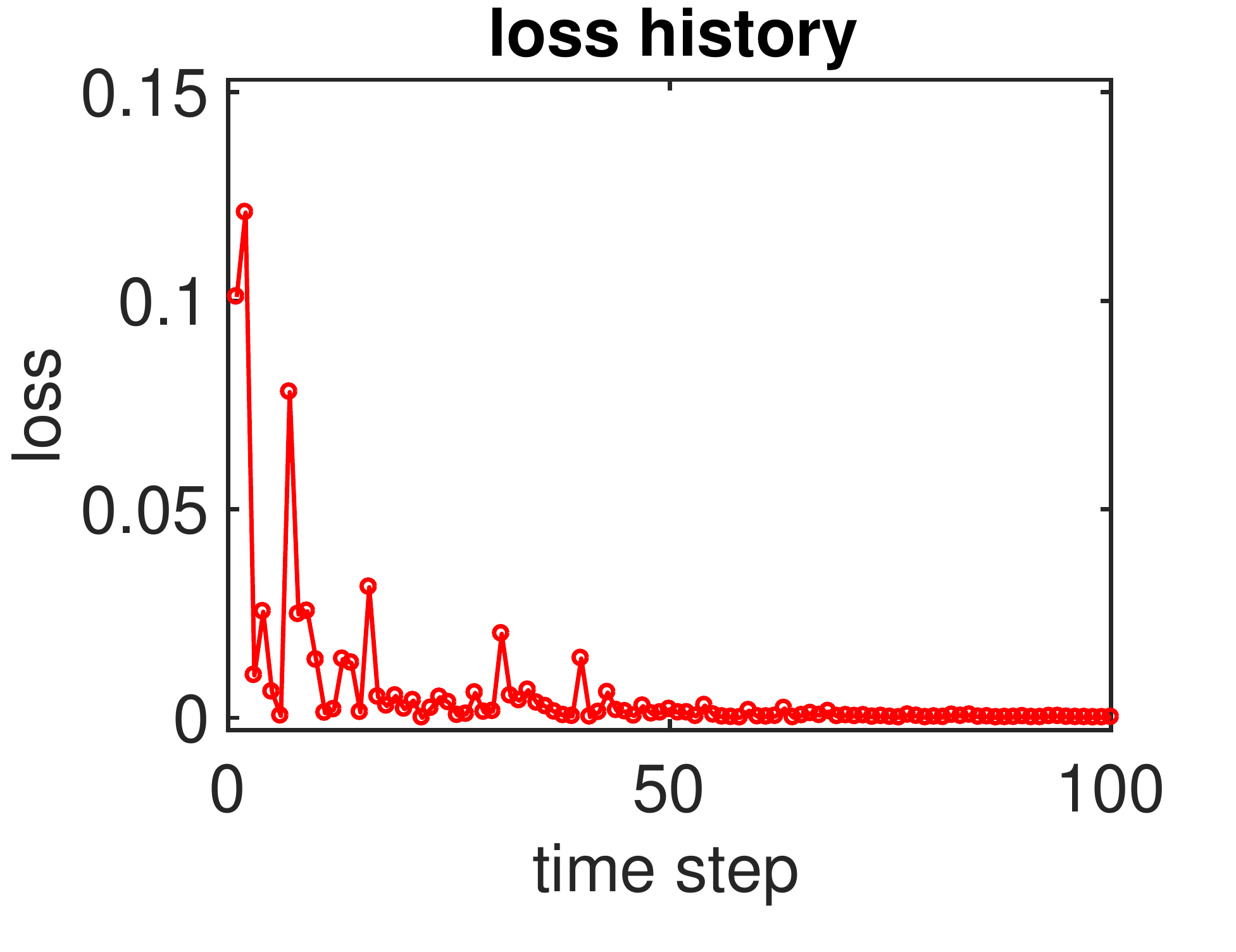}  \\
		(a)&(b)& (c) &(d)
	\end{tabular}
	\caption{(a): The ground truth $f$. (b): The hypothesis $\predf(\cdot,\param_5)$ used for prediction after 4 learning iterations. (c): The final hypothesis $  \predf(\cdot,\param_{100})$. (d): The evolution of the prediction losses $\ell_t(\param_t)$ for $t=1,...,100$.}
\end{figure*}

Some simulation results are depicted in Fig. \ref{fig:onlineregrrbfn}. Note, at the final recorded time step $T= 100$, the hypothesis $\predf(\cdot;\param_t)$ closely matched the ground truth $f(\cdot)$ (cf.  Fig. \ref{fig:onlineregrrbfn}.a and Fig. \ref{fig:onlineregrrbfn}.c ). Furthermore, we note that the prediction losses seemed to  vanish with increasing learning experience  (cf.  Fig. \ref{fig:onlineregrrbfn}.d). 
In the next subsection, we will investigate to what extent such a behavior can be predicted on the basis of the no-regret properties of the learning algorithm we utilised.

\subsection{Prediction loss convergence with increasing permanence} Refer to Lem. \ref{lem:sumStconvtpnewconv}. Its implications for no-regret learning in OPs are clear: while, without further assumptions, from the fact that $\frac 1 T \sum_{t=1}^ T\ell_t(\param_t) \to 0$ alone, we cannot infer that the prediction loss will converge to zero in the traditional sense, we can guarantee convergence with increasing permanence:

\begin{theorem}\label{thm:norawipnoerror}
	Consider the online learning and prediction problem in an OP setting (as for instance considered above) where $\ell_t(\param_t) \geq 0$ is the stage prediction error of the prediction model with parameter $\param_t$ and where the prediction model class is sufficiently expressive to guarantee that $\min_\param \ell_t(\param) =0, \forall t$ (which is the case e.g. when $\hypspace \supseteq \mathcal F$).
	
	If the $\param_t$ are updated with a no-regret algorithm suitable for the given OP then the prediction errors vanish with increasing permanence. That is we have:  $$\ell_t(\param_t) \stackrel{}{ \cwip 0} \text{  (as   } t \to \infty).$$
	\begin{proof}
		Let $\param^* \in \paramspace$ such that $\predfn(\cdot;\param^*) = f(\cdot)$. Thus, $\ell_t(\param^*) =\min_{\param \in \paramspace} \ell_t(\param) =0$.
		Considering that $\ell_t(\param_t) \geq 0 \forall t$, the no-regret guarantee ensures 
		$$\frac{R(T)}{T} = \frac{1} T \sum_{t=1}^T \ell_t(\param_t) - \frac 1 T \sum_{t=1}^T \ell_t(\param^* ) =   \frac{1} T \sum_{t=1}^T \ell_t(\param_t) \stackrel{T\to \infty}{  \longrightarrow 0}.$$ Appealing to Lem. \ref{lem:sumStconvtpnewconv} gives the desired result.
	\end{proof}
\end{theorem}
Note, the theorem applies to the online prediction setup of the previous section, guaranteeing that GP learning of the RBFNNs results in prediction errors that converge to zero with increasing permanence.

The theorem above is valid irrespective of the nature of the time series. However, the assumption that the target f is contained in the hypothesis space is a limitation. In practice, we might have the situation where the hypothesis space and the target class $\mathcal F$ are distinct with a distance given by some \emph{representational model class error} $r =  \sup_{f \in \mathcal F} \inf_{\predf \in \hypspace} \metric(f,\predf)$ for some suitable metric $\metric$ whose choice depends on the concrete online learning problem. In what is to follow we extend our convergence guarantees to two such learning problems in the presence of representational error.
\subsubsection{Online regression with i.i.d. inputs and representational mean-square model error}
Consider online regression. In this standard learning scenario, the inputs are assumed to be drawn i.i.d. from a distribution with density $p$ with support $\inspace$. 
We assume at each stage $t$, the new predictor's parameter $\param_t$ is to be picked based on the history $\mathbb I_t$ of past observations of input -output pairs $(x_i,y_i)$ (or just past prediction losses $\ell_i(\param_i)) \, (i <t)$. After this, $(x_t,y_t)$ and the prediction loss $ \ell(x_t;\param_t)$ can be computed (alternatively, this new loss is revealed) which concludes the stage.  Typically, one considers mean-square regression with a loss $\ell(x;\theta) = \norm{ \predf(x;\theta) - f(x) }^2$.
Let \expect{\cdot} denote the expectation operator.
The expectation $\expect{\ell(x;\theta)}$  is the standard stochastic mean-square loss. Relative to this loss, we can define the model error $r_f =  \inf_{\predf \in \hypspace} \metric(f,\predf)$ for a given function $f \in \mathcal F$ where $\metric(f,g) := \expect{ \norm{f-g}_2^2 }$.  Furthermore, we can consider the worst-case model class error $r_2 := \sup_{f \in \mathcal F} r_f$.
Our theory developed so far assures us that the online mean-square prediction losses i.p.-converge to a set that is below this representational model error: 

\begin{theorem} \label{thm:l2onlinewipconsistency}
	Assume the online regression task is performed by a prediction algorithm that suffers sub-linear external regret, i.e. where the $\param_t$ are chosen such that $\exists \Delta \in o(T) \forall T \in \nat : \sum_{t=1}^T  \ell(x_t;\param_t) \leq \inf_{\param \in \paramspace} \sum_{t=1}^T \ell(x_t;\param) + \Delta(T)$. Then the sequence $(\expect{ \ell(x_t;\param_t)}  )_{t  \in \nat}$ of expected prediction losses converges to at most the representational error with i.p., that is:  $$\expect{ \ell(x_t;\param_t)} \stackrel{t \to \infty}{\cwip}r_f \in  [0,r_2].$$
\end{theorem}
Applied to our example of RBFN-based online regression, the theorem states that the mean-square prediction error of the predictors that are found online i.p. -converges to the best mean-square representational error attainable by the presupposed RBFN structure.

\section{No-Regret Learning-Based Model-Reference Adaptive Control } \label{sec:mracappl}
As mentioned above, our results are meaningful in online-learning based model-reference adaptive control.
Consider a dynamical system $\ddot x = f(x) + a(x) + b(x) u$  where $x$ denotes the state and $u$ denotes the control action. We assume that $a$ and $b$ are known a priori and that the inverse $b^{-1}$ can be computed for all states $x$. By contrast, $f$ is uncertain and capturing model discrepancies due to environmental conditions of the system the plant operates i,  which are hard to model a priori.
If  $f$ was perfectly known, a standard approach would be to feedback linearise the system, setting $u(x):= b^{-1}(x) (- a(x)-f(x) + \ddot x_{ref}) $ where $\ddot x_{ref}$ is some reference behaviour we desire the closed-loop dynamics to exhibit. For example, in tracking where we desire $x_t$ to follow a target trajectory $\xi_t$, one approach would be to set $\ddot x_{ref}= K (\xi-x)$ where $K$ a stabilising feedback matrix ensuring that $\xi_t$ will eventually be tracked by the reference state trajectory $x_{ref}$ with sufficient accuracy. 
In the absence of perfect knowledge of $f$ we can replace $f(x_t)$ by a predictor $\predf(x_t;\param_t)$ in  the feedback-linearising law $u(x_t)$ and to learn the parameters of the predictor online. There are many learning methods we can employ to this end, including  updates of neural network weights \cite{Anderson2009} or nonparametric learning methods such as Gaussian processes \cite{ChowdharyCDC2013}.  In this work, we propose to learn this predictor online with an OP-based online regression method as describedin  in Sec. \ref{sec:OCP4ORegr}.

It can be shown, that when defining the error  by $e = x_{ref} - x$, in a first-order Euler-discretised version, the error dynamics become (see e.g. \cite{calliess2014_thesis}):
$$e_{t+1} = M e_t + d_t.$$ Here, $M$ is a stable matrix, disturbance $d_t = \tinc f(x_t) - \tinc \predf(x_t; \param_t)$ is the prediction error and $\tinc$ is the time increment from the time-discretisation.
Assume we employ online regression to online-learn $\predf$ (by adapting the parameters $\param_t$) with a sublinear regret incurring algorithm OP algorithm. In that case, with some additional assumptions, we can guarantee that the error trajectory vanishes with increasing permanence.
\begin{theorem}\label{thm:noractrlerrcwip}
Suppose that (i) $f$ resides in the model class $\hat \fctspace$ and that (ii) $\sup_{\predf \in \hat \fctspace} \norm{f-\predf}_\infty$ is bounded. If the parameters $\theta_t$ are incrementally updated with an algorithm that incurs sublinear external regret and receives loss feedback $\ell_t(\theta_t) = \norm{d_t}_2^2$ then the error trajectory vanishes with increasing permanence, i.e. $e_t \cwip 0$.
\begin{proof} 
Owing to (i) and the no-regret assumption Thm. \ref{thm:norawipnoerror} is applicable, which entails that $d_t \cwip 0$. In conjunction with (ii), this allows us to appeal to Thm. \ref{thm:stablewipperturbedlindyn_main} which gives the desired statement.
\end{proof}
\end{theorem}

\subsection{Example -- pendulum control}
To illustrate the viability of no-regret learning based control on a simple example we, we consider the following pendulum control problem:

We explored our method's properties in simulations of a rigid pendulum with (a priori known) drift $a(x) :=  - \frac{g}{l} \sin (x_1) - \frac{r(x_1)}{m l^2} x_2$ and constant control input function $b(x) = \frac{1}{m \, l^2}$.  Here, $x_1 = q, x_2 = \dot q \in \Real$ are joint angle position and velocity, $r$ denotes a friction coefficient, $g$ is acceleration due to gravity $l$ is the length and $m$ the mass of the pendulum.  The control input $u \in \Real$ applied a torque to the joint that corresponds to joint-angle acceleration. With $q$ denoting the joint angle, $q=0$ pertains to a state where the pendulum is pointing downward and $q=\pi$ denotes a position in which the pendulum is upward. Given an initial configuration $x_0 = [0;0]$, we desired to steer the state to a terminal configuration $\xi = [\pi, 0]$. We applied a feedback linearising control law corresponding to the continuous-time law $u(x_t;\param_t) := b^{-1}(x_t) (-a(x_t) -\predf(x_t;\param_t)  - K (\xi - x_t) )$. Here $K$ was (an underdamping) PD-controller feedback matrix ensuring global asymptotic stability of $\xi$ in the perfectly feedback linearised reference system $ \ddot x_{ref} = K(\xi-x_{ref})$. To connect to our theory, we discretised the system and control laws by a first-order Euler approximation. We documented the behavior of three different settings: (i) Where $f(\cdot) = \predf(\cdot) =0$, corresponding to a situation where there is no model error. (ii) Where $f(x) = \sum_{i=1}^4 w_i^* \exp(\abs{x_1-c_i}/\sigma_i)$ with $w^*= [-12,-10,10,12],c= [-\pi/2,0,\pi/2,\pi], 
\sigma = [1,1,.5,.5]$ being a mixture of Gaussians while the controller falsely modeled this function by the static predictor $\predf \equiv 0$; (iii) Where $f$ was as before but, at each time t, the predictor was chosen to be $\predf(\cdot;\param_t) = \sum_{i=1}^4 \param_{t,i} \exp(\abs{x_1-c_i}/\sigma_i)$ and parameter $\param_t$ updated with Greedy Projection\cite{zinkevich:icml2003}. The initial parameter was $\param_0= [0,0,0,0]$.

As learning feedback prediction loss $\ell_t(\param_t) = \norm{ \predf(x_t;\param_t) - \bigl ( \ddot x_t -a(x_t) -b(x_t) u(x_t;\param_t) \bigr )}_2^2$ was fed back to the no-regret algorithm after each discrete time step  $t$.

According to our theory the reference error $e_t$ should vanish (at least with increasing permanence).
Simulation results are depicted in Fig. \ref{fig:1pendctrln}. The documented behaviour is consistent with our theoretical guarantees. Note, that the linearising controller without the adaptive element that operates on the basis of the static inaccurate model led to poor performance. By contrast, when the controller was set up with no-regret learning in place as described above (and matching the assumptions of Thm. \ref{thm:l2onlinewipconsistency}) seamlessly managed to control the state to the target.
\begin{figure*}[t]
	\label{fig:1pendctrln}
	\centering
	\begin{tabular}{ccc}
		\includegraphics[height= .2\textwidth, width=0.31\textwidth]{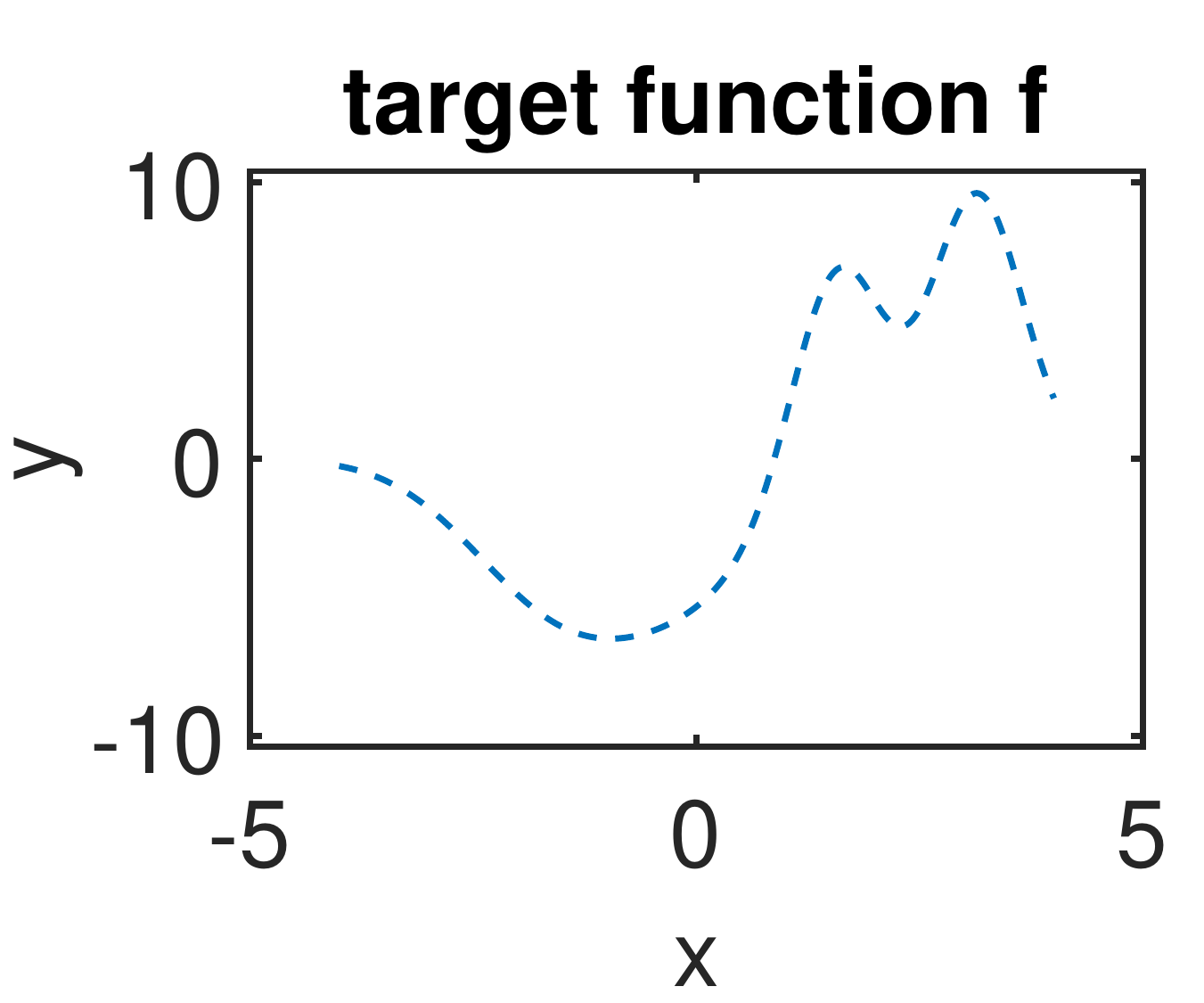}
		&\includegraphics[height= .2\textwidth, width=0.31\textwidth]{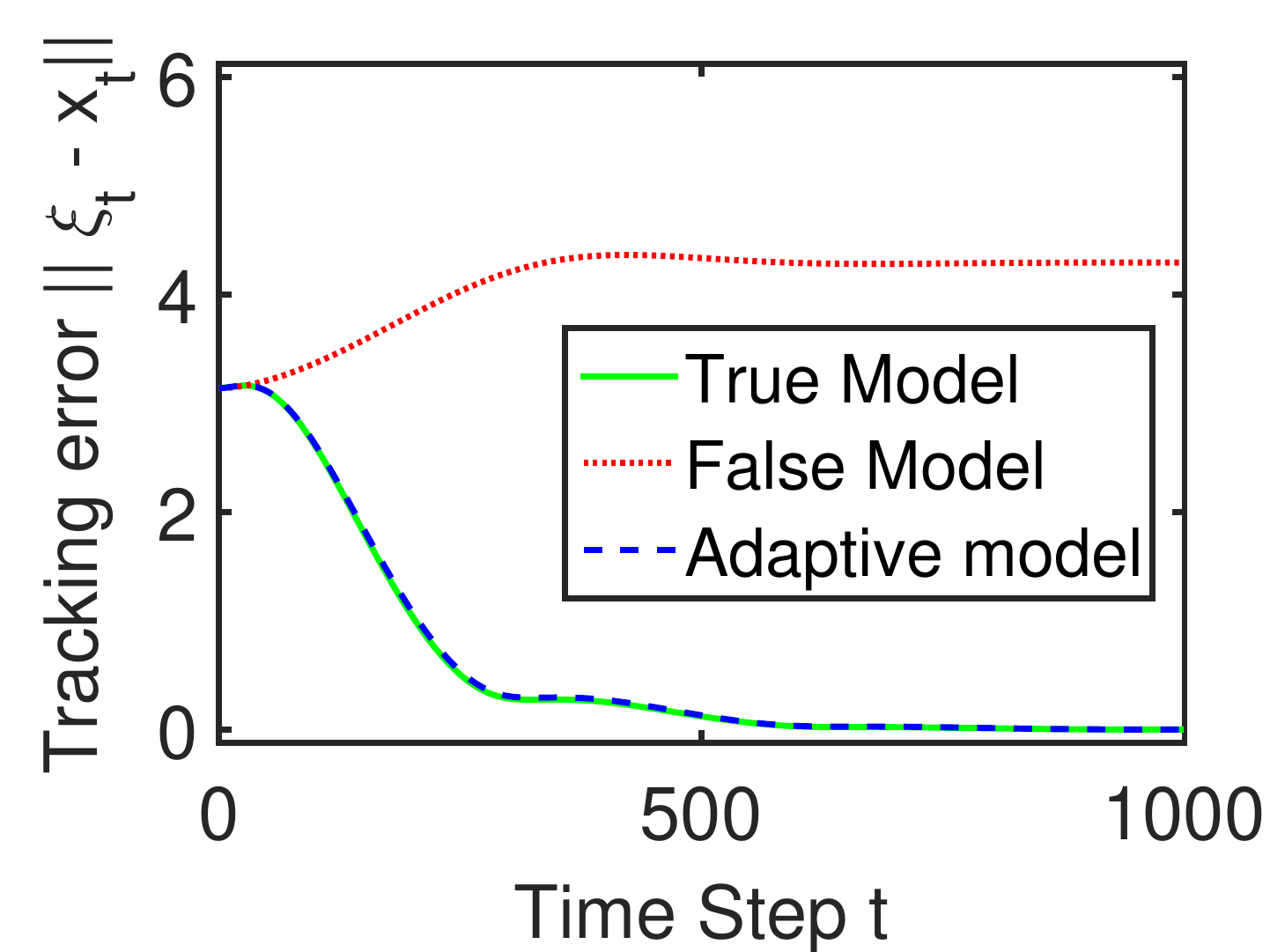} 
		&\includegraphics[height= .2\textwidth, width=0.31\textwidth]{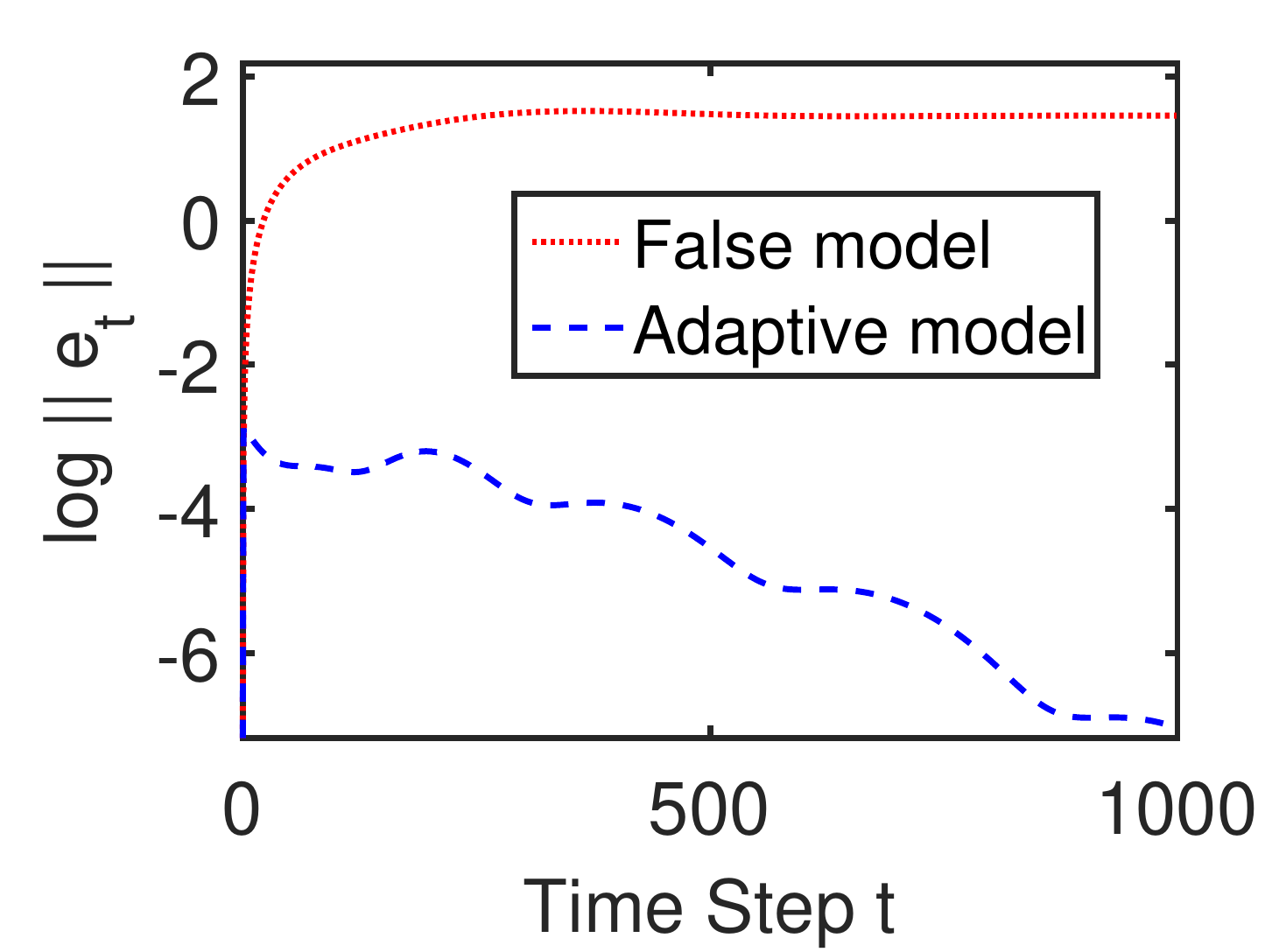}\\
		(a)&(b)& (c)
	\end{tabular}
	\caption{ Pendulum control task. (a): The ground truth model error $f$. (b): Recorded tracking errors incurred by linearsing controllers with knowledge of the true $f$ (green plot), under the false assumption that $f\equiv \predf \equiv 0$ (red plot) and when utilising our no-regret learning based prediction model $\predf(x_t;\param_t)$ with the parameters updated with GP (blue plot). Note, how falsely fixing  $\predf \equiv 0$ caused the second controller to completely fail to track the target. By contrast, our learning-based controller adapted quickly and learned to accurately track the target. (c): Logarithm of the norm of the error dynamics (deviation from the reference) for the linearising controller based on the false model (red plot) and for our learning based approach (blue plot).}
\end{figure*} 

\vspace{-.1em}
\section{Conclusions}
This paper discussed early work considering the application of online programming (in particular, of no-regret algorithms) in online learning-based control settings. 
As a first proposal of how to think about the implications of no-regret guarantees in control, we introduced the concept of convergence with \emph{increasing permanence (i.p.)} as an objective for online learning and control. We discussed conditions under which such  objectives are met when no-regret algorithms are employed in online regression and learning-based model-reference adaptive control. We have given illustrations of such applications of no-regret algorithms to online regression and control tasks. 

Being work at an early stage, many possible directions of further enquiry remain. 
Firstly, i.p.-convergence appears to be a relatively weak notion that may not be satisfactory to be assured of in many control settings. Future work could investigate how regret-bounds can lead to refined convergence rates and lead to stronger guarantees if additional properties are known such as bounded excitation (i.e. target function values), Lipschitz constants and convexity properties.
Secondly, on the practical side, we will devise nonparametric extensions of no-regret learning based controllers that would be able to learn and control dynamics more flexibly. In addition, 
we will investigate of how to systematically deal with observational noise and representational model error both from a practical and a theoretical perspective. 
For simplicity, our exposition was limited to the case of noise-free observations (translating to noise-free loss observations). Fortunately however, there exist a variety of online convex programming algorithms that address noisy losses  \cite{CesaBianchi2010,Belmega2018,Argawal2013,Anava2013} the offer expected regret bounds that can be converted into high-probability bounds. If we applied those, all our results presented in this paper would hold with arbitrarily high probability. 
Furthermore, we would find it interesting to investigate applications to model-predictive control.

In summary, this paper presented a first step towards drawing a bridge between control and the online optimisation community in theoretical machine learning. While many open theoretical challenges remain and details will have to be worked on in the future, we hope to have provided first evidence that utilisation of no-regret learning techniques might be a potentially fruitful direction in the design and analysis of new online-learning based controllers. 

%
\bibliographystyle{plain}
\small
\bibliography{lit}

\section{Supplementary material}

\subsection{Derivation of Lemma \ref{lem:sumStconvtpnewconv}}
\begin{lemma}[ Lem. \ref{lem:sumStconvtpnewconv}]
	Assume we are given a non-negative real-valued sequence $(s_t)_{t \in \nat}$ with \newline
	$ S_T := \frac{1}{T} \sum_{t=1}^T s_t \stackrel{T \to \infty}{\to} 0.$
	Then we have: $s_t \cwipt 0$.
	
	\begin{proof}
		Assume $S_T := \frac{1}{T} \sum_{t=1}^T s_t \stackrel{T \to \infty}{\to} 0$, i.e.
		\begin{equation} \label{eq:oirfhkdjvn}
		\forall e>0 \exists m_e  \forall m\geq m_e : S_m \leq e.
		\end{equation}
		For contradiction assume $\neg(s_t \cwipt 0)$. This would imply the existence of some  $c \in \nat, \epsilon > 0,N \in 		\nat$ and  such that 
		\begin{equation}
		\forall n \geq N \exists i \in \{1,...,c\} : s_{n+i} > \epsilon.
		\end{equation}
		Let $q \in \argmin \bigl \{ m \in \nat \, | \, mc \geq \max(N,m_{\frac{\epsilon}{2c}}) \bigr\}.$ $\forall m >q $ we have:
		\begin{align}
			\frac \epsilon {2c}  \geq  S_{mc} &=  \frac{1}{mc} \sum_{t< qc} s_t + \frac{1}{mc} \sum_{t=q c}^{m c} s_t\geq\frac{1}{mc} \sum_{t=q c}^{m c} s_t  \\
			& \stackrel{(\ref{eq:oirfhkdjvn})} > \frac{(m-q) \epsilon}{mc}
		\end{align}
	which implies $\frac \epsilon {c} > 2(1- \frac q m) \frac \epsilon {c}$ for all $m > q$. In particular, the inequality would hold when choosing $m = 2q$. However, substituting in this choice yields the false inequality $\frac \epsilon {c} > \frac \epsilon {c}$.
	\end{proof}
\end{lemma}
%
\vspace{-1em}

\subsection{Derivation of Thm. \ref{thm:cwipcontrdynsys_main}}
\begin{theorem} \label{thm:cwipcontrdynsys} [ Thm. \ref{thm:cwipcontrdynsys_main}] Let $(\inspace, \norm{\cdot})$ be a normed vector space and $\phi : \inspace \to \inspace$ be a contraction with fixed point $x_* \in \inspace$ and Lipschitz constant $\lambda <1$ relative to the metric canonically induced by norm $\norm{\cdot}$.   Let $(y_n)_{n \in \nat} , (d_n)_{n \in \nat}$  be sequences in $\inspace$ satisfying \begin{equation}\label{eq:contractdyn}
	y_{n+1}  = \phi( y_n) + d_n
\end{equation}
for all time steps $n \in \nat_0$. Assume bounded disturbances, i.e.  $\exists \maxerrn \in \Real \forall n :\norm{d_n} \leq \maxerrn$. Let $r \geq 0$. 

If  $\norm{d_n} \stackrel{n\to \infty}{\cwip} [0,r]$ then we have: $$\norm{ y_n - x_* } \stackrel{n\to \infty}{\cwip} \Bigl [0,  \frac{r}{1- \lambda} \Bigr ].$$

\begin{proof}
Let $\bar x_0 =: y_0$ and define $\bar x_n$ by the recurrence 
\begin{equation}
\label{eq:nominrecurrcontr}
\bar x_{n+1} = \phi(\bar x_n), \, \forall n \in \nat.
\end{equation}

Note, we can define 
$$\sigma := \sum_{i=0}^\infty \lambda^i = \lim_{n \to \infty}  \sum_{i=0}^{n-1} \lambda^{n-1-i} = \frac{1}{1-\lambda} $$
where we have applied the geometric series formula in the last step.

Assume $\norm{d_n} \stackrel{n\to \infty}{\cwip} [0,r]$.
 
 Let $\epsilon >0,D, N \in \nat$. We desire to show: 
	\begin{equation}
	\exists p \geq N \in \nat \forall i \in \{1,...,D\}:   \norm{y_{p+i } - x_*}  \leq \epsilon + \sigma r.
	\end{equation}
 
To this end, firstly, we note that due to convergence of $x_n$ to the fixed point $x_*$, we can find  $n_0$ such that 
\begin{equation} \label{statement:convfpx} \forall n \geq n_0: \norm{\bar x_{n} - x_*}  < \frac \epsilon 3. \end{equation}

Secondly, we note that, 
by induction, it is easy to show that for all $\ k,n \in \nat$, we have 
\begin{align}
 \norm{y_{k+n} - \bar x_{k+n}} & \leq \lambda^n \norm{\bar x_k-y_k} +\sum_{i=0}^{n-1} \lambda^{n-1-i} \norm{d_{k+i}} \\
 &\leq  \lambda^n \norm{\bar x_k-y_k} +\maxerrn_{k,n} \sigma \label{ineq:contrmaxerrnsigma}
\end{align}
where $\maxerrn_{k,n} := \max \{ \norm{d_k},\dots, \norm{d_{k+n-1}} \}$.
Since by assumption, the disturbances are bounded, say by $\maxerrn \in \Real$ the last inequality implies that in particular, 
$ \forall n: \norm{y_{n} - \bar x_{n}} \leq  \lambda^n \norm{\bar x_0-y_0} +\maxerrn \sigma \leq \norm{\bar x_0-y_0} +\maxerrn \sigma< \infty$. 
Hence, we can choose $q_0$ such that 
	\begin{equation}\label{ineq:q0lambdaj}
	\forall j \geq q_0, \forall n:  \lambda^j \norm{\bar x_{n}-y_{n}} \leq \frac{\epsilon}{3}.
	\end{equation}
	
Moreover, the assumption, $d_n \cwip [0,r]$ implies the existence of some $k_0 \geq N-1$ with
	\begin{equation}
	\forall j \in \{1,...,D +n_0 +q_0\}:   \norm{d_{k_0+j}}  \leq \frac{\epsilon}{3 \sigma} +r.
	\end{equation}
	and hence, 
	$\maxerrn_{k_0,n_0+D}	 \leq \frac{\epsilon}{3 \sigma} +r.$

Now, choose $ p:= k_0 +n_0+q_0$, $i \in \{1,\dots,D\}$ and let $j := p +i -k_0 =n_0+q_0+i $. 
We have: 
\begin{equation}\label{ineq:mxerrncomp}
\maxerrn_{k_0,j}\leq \maxerrn_{k_0,n_0+D+q_0}   \leq \frac{\epsilon}{3 \sigma} +r
\end{equation}
and furthermore, we have:

$\norm{y_{p+i} - x_*}  \leq  \norm{x_*- \bar x_{n_0+(k_0+q_0+i)}}+ \norm{y_{k_0+j} - \bar x_{k_0+j}}\\
\stackrel{(\ref{statement:convfpx})}{  \leq} \frac{\epsilon}{3} + \norm{y_{k_0+j} - \bar x_{k_0+j}}
\stackrel{(\ref{ineq:contrmaxerrnsigma})}{  \leq} \frac{\epsilon}{3} + \lambda^j \norm{\bar x_{k_0}-y_{k_0}} +  \maxerrn_{k_0,j} \sigma\\
\stackrel{(\ref{ineq:mxerrncomp} )}{\leq} \frac{\epsilon}{3} + \lambda^j \norm{\bar x_{k_0}-y_{k_0}} +  (\frac {\epsilon}{3 \sigma}+r) \sigma
\stackrel{(\ref{ineq:q0lambdaj})}{  \leq} \frac{\epsilon}{3} + \frac \epsilon 3 +  (\frac {\epsilon}{3 \sigma}+r) \sigma = \epsilon + \sigma r.$
\end{proof}
\end{theorem}
\vspace{-1em}
\subsection{Derivation of Thm. \ref{thm:stablewipperturbedlindyn_main}}
Let $\inspace$ denote state space endowed with a norm $\norm{\cdot}$. 
We consider the discrete-time dynamical system:
\[ x_{n+1} = M x_n + d_n \] 
where for time step $n \in \nat_0$, we refer to $d_n \in \inspace$ as a \emph{disturbance}. 

We assume that the disturbances are bounded and that $M$ is a stable matrix with spectral radius strictly less than 1, i.e $\specrad(M) < 1$ and $\forall n : \norm{d_n} <\maxerrn$  for some upper bound $\maxerrn$ on the disturbance. 
%
%
%
%
%
%
%
By induction, it is easy to show that for all $k \in \nat_0,n \in \nat$ we have $	\vc x_{k+n} =  M^n \, \vc x_{k} +  \sum_{i=0}^{n-1} M^{n-1-i} \, d_{i+k}$. 
%
%
Hence, 
\begin{align}
	\norm{x_{k+n}} &\leq   \matnorm{M^{n}} \, \norm{\vc x_k} +\sum_{i=0}^{n-1}  \matnorm{M^{n-1-i}} \, \norm{d_{i+k}}		\\
	&\leq \matnorm{M^{n}} \, \norm{\vc x_{k+n}} + \maxerrn_{k,n}	 \sum_{i=0}^{n-1}  \matnorm{M^{n-1-i}}\label{ineq:errnormbasic}
\end{align}
where $\matnorm{\cdot}$ denotes the spectral norm $\maxerrn_{k,n}	= \max_{i=0,...,n-1} \norm{d_{i+k}} \leq \maxerrn$.
Since $M$ is stable, the terms in (\ref{ineq:errnormbasic}) are bounded and convergent as $n \to \infty$ (see e.g. \cite{calliess2014_thesis}). In particular, with Gelfand's formula and the standard root test for series it is easy to establish convergence of the series: That is, there exists $\sigma \in \Real$ with $\lim_{k \to \infty} \sum_{i=0}^{k-1}  \matnorm{M^{k-1-i}} = \sum_{i=0}^{\infty}  \matnorm{M^{i}} =:\sigma$.\footnote{In \cite{calliess2014_thesis}, a practically computable upper bound on $\sigma$ can be found. } And, we have $ \sum_{i=0}^{n-1}  \matnorm{M^{n-1-i}} \leq \sigma, \forall n$.
Hence,
\begin{equation}
\norm{\vc x_{k+n}} \leq \matnorm{M^{n}} \, \norm{\vc x_k} + \sigma \, \maxerrn_{k,n}, \forall n \in \nat, k \in \nat_0.
\label{eq:errdyn320894}
\end{equation}
\begin{theorem}[Thm. \ref{thm:stablewipperturbedlindyn_main}]\label{thm:stablewipperturbedlindyn}
	If the sequence of disturbances vanishes with increasing permanence up to error $r>0$, i.e. $\norm{d_n} \cwip [0,r]$ then
	\[\norm{x_n} \stackrel{n \to \infty}{\cwip} [0,\sigma r]. \]
	
\end{theorem}
\begin{proof} 
	Let $\epsilon >0,D, N \in \nat$. We desire to show: 
	\begin{equation}
	\exists p \geq N \in \nat \forall i \in \{1,...,D\}:   \norm{x_{p+i}}  \leq \epsilon + \sigma r.
	\end{equation}
	%
	
	Since we assumed the disturbances to be bounded, the sequence \seq{x_n}{n\in \nat} is bounded (cf. (\ref{eq:errdyn320894})).  That is, $\exists \beta \in \Real \forall k : \norm{x_k} \leq \beta$.
	Hence, $\matnorm{M^{n}} \, \norm{\vc x_k} \leq \matnorm{M^{n}} \beta \stackrel{n \to \infty}{\longrightarrow} 0$. Here, the convergence to zero follows from the assumption that $M$ is a stable matrix and, owing to boundedness, convergence is uniform in the following sense:
	
	\begin{equation}\exists n_0 \in \nat \forall n \geq n_0 \forall k : \matnorm{M^{n}} \, \norm{x_k} \leq \frac\epsilon 2. 
	\label{expr:1}\end{equation}

	Moreover, the assumption, $d_n \cwip [0,r]$ implies that we can choose some $m \geq N-1$ with
	\begin{equation}
	\forall i \in \{1,...,D +n_0\}:   \norm{d_{m+i}}  \leq \frac{\epsilon}{2 \sigma} +r.
	\end{equation}
	and hence, 
	\begin{equation}\label{eq:Fnconvjjff}
	\maxerrn_{m,n_0+D}	 \leq \frac{\epsilon}{2 \sigma} +r.
	\end{equation} 

	Now, choose $ p:= m +n_0$. Then for all $i \in \{1,\dots,D\}$:
	\begin{align}
		\norm{\vc x_{p+i}}  =\norm{\vc x_{m+(n_0+i)}} &\stackrel{(\ref{eq:errdyn320894})}{\leq} \matnorm{M^{n_0+i}} \, \norm{\vc x_m} + \sigma \, \maxerrn_{m,n_0+i} \\
		&\stackrel{(\ref{expr:1}),(\ref{eq:Fnconvjjff})}{\leq} \frac \epsilon 2 + \sigma \frac{\epsilon}{2 \sigma} + \sigma r = \epsilon +  \sigma r.
	\end{align}
	
\end{proof}

\subsection{Derivation of Thm. \ref{thm:l2onlinewipconsistency} }
Consider online regression. In this standard learning scenario, the inputs are assumed to be drawn i.i.d. from a distribution with density $p$ on support $\inspace$. 
We assume at each stage $t$, the new predictor's parameter $\param_t$ is to be picked based on the history $\mathbb I_t$ of past observations of input -output pairs $(x_i,y_i)$ (or just past prediction losses $\ell_i(\param_i)) \, (i <t)$. After this, $(x_t,y_t)$ and the prediction loss $ \ell(x_t;\param_t)$ can be computed (alternatively, this new loss is revealed) which concludes the stage.  Typically, one considers mean-square regression with a loss $\ell(x;\theta) = \norm{ \predf(x;\theta) - f(x) }^2$. Its expectation $\expect{\ell(x;\theta)}$  is the standard stochastic mean-square loss. Relative to this loss, we can define the model error $r_f =  \inf_{\predf \in \hypspace} \metric(f,\predf)$ for a given function $f \in \mathcal F$ where $\metric(f,g) := \expect{ \norm{f-g}_2^2 }$.  Furthermore, we can consider the worst-case model class error $r_2 := \sup_{f \in \mathcal F} r_f$.

\textbf{Assumption 1:}
\emph{ We assume the online regression task is performed by a prediction algorithm that suffers sublinear external regret, i.e. where the $\param_t$ are chosen such that $\exists \Delta \in o(T) \forall T \in \nat : \sum_{t=1}^T  \ell(x_t;\param_t) \leq \inf_{\param \in \paramspace} \sum_{t=1}^T \ell(x_t;\param) + \Delta(T)$. }
\begin{lemma} \label{lem:mswipconv1} Let $ r_f =\inf_\param \expect{ \ell(x; \param)}_{x} =\inf_\param  \int_\inspace \ell(x;\param) \d p(x)$. Under Assumption 1, we can guarantee that there exists a sequence $(V_T)_{T \in \nat}$ of random variables  with $V_T \to 0$ a.s. such that we have $$\frac 1 T \sum_{t =1}^T \ell(x_t;\theta_t)  \leq V_T + r_f.$$
	\begin{proof}
		For  each $T \in \nat$ the random function $\psi_T : \param \mapsto \frac 1 T \sum_{t=1}^T \ell(x_t;\param)$ (the joint outcome space is given by the outcome space of the random input sequence $(x_t)_{t \in \nat}$). By the strong law of large numbers we know that $\psi_T(\param) \stackrel{T \to \infty}{\rightarrow} \psi(\param) := \int_\inspace \ell(x;\param ) dp(x) = \expect{\ell(x;\param)} \forall \param$ almost surely (a.s.). Examining the formal $\epsilon-\delta-$ definitions of convergence, its is easy to see that thereby, we also have $\inf_\param \psi_T(\param) \stackrel{T \to \infty}{\rightarrow} \inf_\param \psi(\param)= r_f$ a.s.. Hence, with $V_T:= \max\{ r_f, \inf_\param \psi_T(\param)\} + \frac {\Delta(T)}{ T}  \geq\inf_\param \psi_T(\param)+ \frac {\Delta(T)}{T} $ we have $V_T \to 0$ a.s. and  $\frac 1 T \sum_{t =1}^T \ell(x_t;\param_t) \leq \inf_\param \psi_T(\param) + \frac {\Delta(T)}{T} \leq V_T + r_2$.
	\end{proof}
\end{lemma}
The next lemma shows that the sequence of expected losses is nonnegative.
\begin{lemma}\label{lem:mswipconv2}
	With assumptions and definitions as before, we have 
	\[
	\forall t \in \nat: \expect{\ell(x_t;\param_t) | \param_t,\mathfrak  O_t}_{x_t} \geq r_f  \hspace{1em}\text{and} \hspace{1em}\expect{\ell(x_t;\param_t) |\mathfrak  O_t}_{x_t,\param_t} \geq r_f.\]
	\begin{proof}
		Since the inputs are drawn independently, we have $p(x_t  |\param_t, \mathbb I_t) = p(x_t)$ and hence,
		$\forall t, \param_t: \expect{\ell(\cdot;\param_t) | \param_t; \mathbb I_t}_{x_t} = \int_\inspace \ell(x_t;\param_t) dp(x_t| \param_t , \mathbb I_t) = \int_\inspace \ell(x_t;\param_t) \d p(x_t ) \geq \inf_\param \int_\inspace \ell(x;\param)\d p(x) = r_f$.
		Thus,  also: $\expect{\ell(x_t;\param_t) |\mathfrak  O_t}_{x_t,\param_t} = \int_{\param_t} \expect{ \ell(x_t) | \param_t, \mathbb I_t}_{x_t} \d p(\param_t | \mathbb I_t) \geq \int_{\param_t} r_f \d p(\param_t | \mathbb I_t) = r_f$.
	\end{proof}
\end{lemma}

\begin{theorem}[Thm. \ref{thm:l2onlinewipconsistency}]
	Assume the online regression task is performed by a prediction algorithm that suffers sub-linear external regret, i.e. where the $\param_t$ are chosen such that $\exists \Delta \in o(T) \forall T \in \nat : \sum_{t=1}^T  \ell(x_t;\param_t) \leq \inf_{\param \in \paramspace} \sum_{t=1}^T \ell(x_t;\param) + \Delta(T)$. Then the sequence $(\expect{ \ell(x_t;\param_t)}  )_{t  \in \nat}$ of expected prediction losses i.p.-converges to at most the representational error, that is:  $$\expect{ \ell(x_t;\param_t)} \stackrel{t \to \infty}{\cwip}r_f \in  [0,r_2].$$
	
	\begin{proof}
		For $t \in \nat$ let $s_t: = \expect{\ell(x_t;\param_t) | \mathbb I_t} - r_f$, 
		which by Lem. \ref{lem:mswipconv2} are known to be nonnegative. 
		By Lem. \ref{lem:mswipconv2} we know $\frac 1 T \sum_{t =1}^T \ell(x_t;\theta_t)  \leq V_T + r_f $ for some a.s. vanishing $V_T$. Hence,
		$$0 \leq \expect{\frac 1 T \sum_{t =1}^T (\ell(x_t;\theta_t) - r_f ) } = \frac 1 T  \sum_{t=1}^T s_t \leq \expect{V_T} \stackrel{T \to \infty} \rightarrow 0.$$ Appealing to Lem.  \ref{lem:sumStconvtpnewconv} allows us to conclude $s_t \cwipt 0$.
	\end{proof}
\end{theorem}
Applied to our example of RBFN-based online regression, the theorem states that the mean-square prediction error of the predictors that are found online converges with i.p. to the best mean-square representational error attainable by the presupposed RBFN structure.
\vspace{-1em}

\end{document}